\documentclass{amsart}

\usepackage{amsmath, amssymb, xcolor}
\input xy
\xyoption{all}

\newtheorem{theorem}{Theorem}
\newtheorem{lemma}[theorem]{Lemma}
\newtheorem{corollary}[theorem]{Corollary}
\newtheorem{fact}[theorem]{Fact}

\theoremstyle{definition}

\newtheorem{remark}[theorem]{Remark}

\def\ot{\operatorname{OT}}
\def\corr{\operatorname{Corr}}
\def\id{\operatorname{id}}
\def\pr{\operatorname{pr}}

\def\Ind#1#2{#1\setbox0=\hbox{$#1x$}\kern\wd0\hbox to 0pt{\hss$#1\mid$\hss}
\lower.9\ht0\hbox to 0pt{\hss$#1\smile$\hss}\kern\wd0}

\def\Notind#1#2{#1\setbox0=\hbox{$#1x$}\kern\wd0\hbox to 0pt{\mathchardef
\nn=12854\hss$#1\nn$\kern1.4\wd0\hss}\hbox to
0pt{\hss$#1\mid$\hss}\lower.9\ht0 \hbox to
0pt{\hss$#1\smile$\hss}\kern\wd0}

\begin{document}

\title{A note on subvarieties of powers of $\ot$-manifolds}

\author{Rahim Moosa}
\address{Rahim Moosa\\
University of Waterloo\\
Department of Pure Mathematics\\
200 University Avenue West\\
Waterloo, Ontario \  N2L 3G1\\
Canada}
\email{rmoosa@uwaterloo.ca}

\author{Matei Toma}
\address{Matei Toma\\
Institut de Math\'ematiques Elie Cartan \\
Universit\'e de Lorraine, B.P. 70239\\
54506 Vandoeuvre-l\`es-Nancy Cedex\\ 
France}
\email{Matei.Toma@univ-lorraine.fr}

\thanks{R. Moosa was partially supported by an NSERC Discovery Grant. M. Toma was partially supported by the ANR project MNGNK, decision N° ANR-10-BLAN-0118.} 

\begin{abstract}
It is shown that the space of finite-to-finite holomorphic correspondences on an $\ot$-manifold is discrete.
When the $\ot$-manifold has no proper infinite complex-analytic subsets, it then follows by known model-theoretic results that its cartesian powers have no interesting complex-analytic families of subvarieties.
The methods of proof, which are similar to [Moosa, Moraru, and Toma ``An essentially saturated surface not of K\"ahler-type", {\em Bull. of the LMS}, 40(5):845--854, 2008], require studying finite unramified covers of $\ot$-manifolds.

\end{abstract}

\date{\today}

\maketitle

\section{introduction}

\noindent
This note is concerned with complex-analytic families of subvarieties in cartesian powers of the compact complex manifolds introduced by Oeljeklaus and the second author in~\cite{ouljeklaustoma}, here referred to as $\ot$-manifolds.
These manifolds are higher dimensional analogues of Inoue surfaces of type $S_M$.
In~\cite{moosamorarutoma}, we, along with Ruxandra Moraru, showed that if $X$ is an Inoue surface of type $S_M$ then $X^n$ contains no infinite complex-analytic families of subvarieties, except for the obvious ones such as $\big(\{a\}\times V:a\in X^m\big)$ where $V$ is a fixed subvariety of~$X^{n-m}$.
Using model-theoretic techniques we were able to reduce the problem to considering only the case of $n=2$.
That case amounted to showing that the set of finite-to-finite holomorphic correspondences on $X$, viewed as subvarieties of $X^2$, is discrete.
Here we extend this result to $\ot$-manifolds in general.
Actually, it is useful to consider the following higher arity version of correspondences: for any compact complex manifold $X$, let $\corr_n(X)$ denote the set of irreducible complex-analytic $S\subset X^n$ such that the co-ordinate projections $\pr_i:S\to X$ are surjective and finite for all $i=1,\dots,n$.
So $\corr_2(X)$ is the set of finite-to-finite holomorphic correspondences.\footnote{It may be worth pointing out that the elements of $\corr_n(X)$ are simply components of intersections of pull-backs of finite-to-finite holomorphic correspondences.
That is, for $n>1$, if $S\in\corr_n(X)$ and $\pi_i:X^n\to X^2$ is the co-ordinate projection $(x_1,\dots,x_n)\mapsto (x_1,x_i)$, for $i=2,\dots, n$, then each $\pi_i(S)\subset X^2$ is a correspondence and $S$ is an irreducible component of $\displaystyle \bigcap_{i=2}^n\pi_i^{-1}(\pi_i(S))$. This is an easy dimension calculation, see~\cite[Lemma~3.2]{moosapillay12}.}

\begin{theorem}
\label{corrdiscrete}
If $X$ is an $\ot$-manifold then $\corr_n(X)$ is discrete for all $n>0$. 
\end{theorem}

The proof, which we will give in Section~\ref{theproof},  follows to some extent what was done for Inoue surfaces of type $S_M$ in~\cite{moosamorarutoma}.
But this approach leads naturally to the consideration of finite unramified coverings of $\ot$-manifolds, and the latter
are not formally instances of the original construction in~\cite{ouljeklaustoma}.
However,
we show in  Section~\ref{generalisedot} that a mild generalisation of that construction leads to a class of manifolds which is closed   under finite unramified coverings. We call these manifolds also $\ot$-manifolds and the theorem is valid for this larger class.

The theorem is particularly significant when $X$ has no proper positive dimensional subvarieties, because  of the following fact coming from model theory.

\begin{fact}
\label{allcorr}
Suppose $X$ is a compact complex manifold that is not an algebraic curve, is not a complex torus, and has no proper infinite complex-analytic subsets.
Then every 
irreducible complex-analytic subset of a cartesian power of $X$ is
a cartesian product of points and elements of $\corr_n(X)$ for various $n>0$.
\end{fact}

\begin{proof}
This is Proposition~5.1 of~\cite{pillayscanlon2003} together with Lemma~3.3(b) of~\cite{moosapillay12}.
\end{proof}

That $\ot$-manifolds without proper positive dimensional subvarieties are ubiquitous in all dimensions follows from work of Ornea and Verbitsky~\cite{ornea-verbitsky} showing that we get examples whenever $X$ is the $\ot$-manifold corresponding to a number field that has precisely two complex embeddings which are not real.

Putting together the Theorem and the Fact, we conclude:

\begin{corollary}
\label{nofamilies}
Suppose $X$ is an $\ot$-manifold that has no proper infinite complex-analytic subsets.
Then, for all $n>0$, $X^n$ has no infinite complex-analytic families of subvarieties that project onto each co-ordinate.
\end{corollary}

\begin{remark}
The model theorist should note that for $X$ to have no proper infinite complex-analytic subsets is exactly {\em strong minimality} of $X$ as a first-order structure in the language of complex-analytic sets.
Strongly minimal $\ot$-manifolds are of {\em trivial $\operatorname{acl}$-geometry} by the manifestation of the Zilber trichotomy in this context.
By~\cite[Proposition~3.5]{moosapillay12}, the discreteness of $\corr_2(X)$ implies that strongly minimal $\ot$-manifolds are {\em essentially saturated} in the sense of~\cite{sat}.
In particular, we obtain in every dimension examples of essentially saturated manifolds that are not of K\"ahler-type.
This was the original motivation for both~\cite{moosamorarutoma} and the current note.
\end{remark}

\bigskip
\section{Finite covers of $\ot$-manifolds}
\label{generalisedot}

\noindent
We will quickly review the original construction of $\ot$-manifolds from~~\cite{ouljeklaustoma} and then describe how to generalise it. 

Fix a number field $K$ admitting $n=s+2t$ distinct embeddings into $\mathbb C$, which we will denote by $\sigma_1,\dots,\sigma_n$ where $\sigma_1,\dots,\sigma_s$ are real and each $\sigma_{s+i}$ is complex conjugate to $\sigma_{s+i+t}$. Assume that $s$ and $t$ are positive. By Dirichlet's Theorem the multiplicative group of units $\mathcal O_K^{*}$ of the ring of integers $\mathcal O_K$ of $K$ has rank $s+t-1$. The subgroup 
$$\mathcal O_K^{*,+}:=\{a\in\mathcal O_K^*:\sigma_i(a)>0\text{ for all }1\leq i\leq s\}$$
of ``positive" units is free abelian of finite index in $\mathcal O_K^{*}$. 
 Let $U$ be a rank $s$ subgroup of $\mathcal O_K^{*,+}$
that is admissible for $K$ in the sense of~\cite{ouljeklaustoma}.
With respect to the natural action of $U$ on the additive group $\mathcal O_K$, consider the semidirect product $\Gamma=U\ltimes\mathcal O_K$.
Let $m=s+t$ and consider the action of $\Gamma$ on $\mathbb C^m$ given by,
$$(a,x)(z_1,\dots,z_m):=\big(\sigma_1(ax)+\sigma_1(a)z_1,\dots,\sigma_m(ax)+\sigma_m(a)z_m\big).$$
As $U<\mathcal O_K^{*,+}$, this action leaves $\mathbb H^s\times \mathbb C^t$ invariant, and the admissibility condition is equivalent to the action being proper and discontinuous.
The original $\ot$-manifold, denoted by $X(K,U)$, is the quotient of $\mathbb H^s\times \mathbb C^t$ by this action. In the sequel we will denote these manifolds by $X(\mathcal O_K, U)$ in order to distinguish them from their generalisations.

The above construction is generalised by replacing the role of $\mathcal O_K$ in $\Gamma$ by any rank $n$ additive subgroup $M\leq\mathcal O_K$ that is stable under the action of $U$.
We say then that {\em $U$ is admissible for $M$}.
Taking $\Gamma=U\ltimes M$, we again get a proper and discontinuous action on $\mathbb H^s\times \mathbb C^t$, and the quotient is denoted by $X(M,U)$.
We will continue to call  these compact complex manifolds  {\em  $\ot$-manifolds}.
To avoid confusing them with the previous construction we will occasionally say that  they are of type $X(M,U)$ (otherwise of type $X(\mathcal O_K, U)$).
Note that the possibility of generalising the original construction by replacing $\mathcal O_K$ with an order of $K$ is already mentioned in~\cite{ouljeklaustoma}. However only the $\mathbb Z$-submodule structure of $M$ and the stability under the $U$-action are necessary to make the construction work.
   
The universal cover of $X(M,U)$ is  $\mathbb H^s\times \mathbb C^t$ and the fundamental group is $U\ltimes M$.
As the latter is of finite index in $U\ltimes\mathcal O_K$, we see that
$X(M,U)$ is a finite unramified covering of $X(\mathcal O_K,U)$.
In fact, all finite unramified covers are of this form:

\begin{lemma}
\label{ot-finitecover}
The class of $\ot$-manifolds of type $X(M,U)$ is closed under finite unramified coverings.
\end{lemma}

\begin{proof}
Given $X(M,U)$, such a covering would correspond to a finite index subgroup $\Gamma_1\leq U\ltimes\ M$.
Taking $U_1$ to be the image of $\Gamma_1$ in $U$, and setting $M_1:=\Gamma_1\cap M$, it is not hard to check that $U_1$ is admissible for $M_1$ and that the covering is nothing other than $X(M_1,U_1)$.
\end{proof}

Much of the theory of $\ot$-manifolds developed in~\cite{ouljeklaustoma} goes through in this more general setting. In particular,

\begin{lemma}
\label{generalisedot-nvf}
If $X=X(M,U)$ is an  $\ot$-manifold then $H^0(X,T_X)=0$.
\end{lemma}

\begin{proof}
For $\ot$-manifolds of type $X(\mathcal O_K,U)$ this is Proposition~2.5 of~\cite{ouljeklaustoma}.
Imitating that argument, it suffices to prove for $M$ a rank $n$ additive subgroup of $\mathcal O_K$, that the image of $M$ in $\mathbb R^s$ under $(\sigma_1,\dots,\sigma_s)$ is dense.
But this is the case because $M$ has finite index in $\mathcal O_K$ and the latter does have dense image (see the proof of Lemma~2.4 of~\cite{ouljeklaustoma}).
\end{proof}

The following remarks serve as further evidence that the above extension of the definition of $\ot$-manifolds is natural. 

\begin{remark}
{\em Any $\ot$-manifold of type $X(M,U)$ admits a finite unramified  cover of type $X(\mathcal O_K,U)$.}

Indeed, since $M$ is of maximal rank in $\mathcal O_K$, there exists a positive integer $l$ such that $l\mathcal O_K\subset M$. Thus $X(l\mathcal O_K,U)$ is a finite unramified cover of $X(M,U)$. But the multiplication by $l$ at the level of  $\mathbb H^s\times \mathbb C^t$ conjugates the actions of $U\ltimes\mathcal O_K$ and of $U\ltimes l\mathcal O_K$ and thus induces an isomorphism between  $X(\mathcal O_K,U)$ and $X(l\mathcal O_K,U)$.
\end{remark}

\begin{remark}
{\em When $s=t=1$ the class $\ot$-manifolds of type $X(M,U)$ coincides with the class of Inoue surfaces of type $S_M$ defined in \cite{I}. }

Indeed, if one starts with the manifold $X(M,U)$, then choosing a generator $a$ of $U$ with $\sigma_1(a)>1$ and a base $(\alpha_1,\alpha_2,\alpha_3)$ of $M$ over $\mathbb Z$ one obtains a matrix $A(a)\in GL(3,\mathbb Z)$ which represents the action of $a$ on $M$ with respect to this basis. Applying the embedding $\sigma_k$ to the relation 
$a(\alpha_1,\alpha_2,\alpha_3)^\top=A(a)(\alpha_1,\alpha_2,\alpha_3)^\top$ shows that $(\sigma_k(\alpha_1),\sigma_k(\alpha_2),\sigma_k(\alpha_3))^\top$ is an eigenvector of $A(a)$ associated to the eigenvalue $\sigma_k(a)$. In particular this implies $A(a)\in SL(3,\mathbb Z)$ since  $\sigma_1(a)>0$. At this point one sees that $X(M,U)$ coincides with the surface $S_{A(a)}$ as defined in \cite{I}. 

Conversely, starting with any matrix $A\in SL(3,\mathbb Z)$, with one real eigenvalue larger than $1$ and two complex non-real  eigenvalues, we denote by $K$ the splitting field of the characteristic polynomial $\chi_A$ of $A$ over $\mathbb Q$. Then there exists an element $a\in\mathcal O^{*,+}_K$ such that the eigenvalues of $A$ (i.e the roots of $\chi_A$) are precisely $\sigma_1(a), \sigma_2(a), \sigma_3(a)$. We find now an eigenvector $v\in \mathbb Z[\sigma_1(a)]^3$ associated to $\sigma_1(A)$ by solving the system $(A-\sigma_1(a)I_3)v^\top=0$ over $K$. There exist now elements $\alpha_1,\alpha_2,\alpha_3\in \mathcal O_K$ such that $v=(\sigma_1(\alpha_1),\sigma_1(\alpha_2),\sigma_1(\alpha_3))$. Moreover  $\alpha_1,\alpha_2,\alpha_3$ are linearly independent over $\mathbb Q$ since a linear relation would entail a linear relation between the components $v_1,v_2,v_3$ of $v$ over $\mathbb Q$, which combined with the equations $(A-\sigma_1(a)I_3)v^\top=0$ would show that $\sigma_1(a)$ is quadratic over $\mathbb Q$. Now choosing $M$ to be the $\mathbb Z$-sumbodule of $K$ generated by  $\alpha_1,\alpha_2,\alpha_3$ and $U$ the multiplicative group generated by $a$ we get again
$X(M,U)=S_{A(a)}$.
\end{remark}

\bigskip
\section{The Proof}
\label{theproof}

\noindent
As in the case of Inoue surfaces of type $S_M$ studied in~\cite{moosamorarutoma}, we will make use of some deformation theory to prove the main theorem.
But we will need a bit more than was used in~\cite{moosamorarutoma}.
We say that a holomorphic map $f:V\to W$ between compact complex manifolds is {\em rigid over $W$} if there are no nontrivial deformations of $f$ that keep $W$ fixed.
More precisely: Whenever $\mathcal V\to D$ is a proper and flat holomorphic map of compact complex varieties with $V=\mathcal V_d$ for some $d\in D$, and $\mathcal F:\mathcal V\to D\times W$ is a holomorphic map over $D$ with $\mathcal F_d=f$, then there is an open neighbourhood $U$ of $d$ in $D$ and a diagram
$$\xymatrix{
\mathcal V_U\ar[rd]\ar[drrr]^{\mathcal F_U}\ar[dd]_\phi\\
&U&&U\times W\ar[ll]\\
U\times V\ar[ur]\ar[urrr]_{\id_U\times f}
}$$
where $\phi$ is a biholomorphism.
In particular $\mathcal F_s(\mathcal V_s)=f(V)$ for all $s\in U$.

\begin{fact}[Section~3.6 of~\cite{namba}]
\label{deformation}
Suppose $f:V\to W$ is a holomorphic map between compact complex manifolds such that
\begin{itemize}
\item
$H^0(V,f^*T_W)=0$, and
\item
$f_*:H^1(V,T_V)\to H^1(V,f^*T_W)$ is injective.
 \end{itemize}
 Then $f$ is rigid over $W$.
 \end{fact}

\begin{lemma}
\label{rigid}
Suppose $X$ and $Y$ are compact complex manifolds, $H^0(Y,T_Y)=0$, and $f:Y\to X^n$ is a holomorphic map such that $\pr_i\circ f:Y\to X$ is a finite unramified cover for each $i=1,\dots,n$.
Then $f$ is rigid over $X^n$.
\end{lemma}

\begin{proof}
Note that here $\pr_i:X^n\to X$ is the projection onto the $i$th co-ordinate.
Let $f_i:=\pr_i\circ f:Y\to X$.
As each $f_i$ is unramified, we have that
$$\displaystyle f^*T_{X^n}=f^*\left(\bigoplus_{i=1}^n\pr_i^*T_X\right)=\bigoplus_{i=1}^nf_i^*T_X=\bigoplus_{i=1}^nT_Y$$
Hence,
$\displaystyle H^0(Y,f^*T_{X^n})=\bigoplus_{i=1}^nH^0(Y,T_Y)=0$.
On the other hand, the isomorphism $(f_1)_*:H^1(Y,T_Y)\to H^1(Y,f_1^*T_X)$ factors through $f_*:H^1(Y,T_Y)\to H^1(Y,f^*T_{X^n})$, and hence the latter is injective.
So $f:Y\to X^n$ is rigid over $X^n$ by Fact~\ref{deformation}.
\end{proof}

We can now prove the main theorem.

\begin{proof}[Proof of Theorem~\ref{corrdiscrete}]
Suppose $X$ is an $\ot$-manifold of type $X(M,U)$.
As in~\cite{moosamorarutoma}, in order to show that $\corr_n(X)$ is discrete we let $S\in\corr_n(X)$ be arbitrary, consider the irreducible component  $D$ of the Douady space of $X^n$ in which $S$ lives, and show that $D$ is zero-dimensional.
This suffices as it proves that each element of $\corr_n(X)$ is isolated in the Douady space.

Let $Z\subset D\times X^n$ be the restriction of the universal family to $D$.
By the flatness of $Z\to D$, for general $d\in D$, $Z_d\in\corr_n(X)$ also.
Let $\widetilde Z\to Z$ be a normalisation and denote by $f:\widetilde Z\to D\times X^n$ the composition of the normalisation with the inclusion of $Z$ in $D\times X^n$.
Then for general $d\in D$ we have that $f_d:\widetilde Z_d\to X^n$ is such that each projection $\pr_i\circ f_d:\widetilde Z_d\to X$ is a finite surjective map.
In~\cite{battistioeljeklaus} it is shown that $\ot$-manifolds of type $X(\mathcal O_K, U)$, and hence also $\ot$-manifolds of type $X(M,U)$, have no divisors.
So the purity of branch locus theorem (which applies as $\widetilde Z_d$ is normal and $X$ is smooth) implies that $\pr_i\circ f_d$ is a finite unramified covering.
In particular, $\widetilde Z_d$ is a generalised $\ot$-manifold by Lemma~\ref{ot-finitecover}, and so $H^0(\widetilde Z_d,T_{\widetilde Z_d})=0$ by Lemma~\ref{generalisedot-nvf}.
But moreover, by Lemma~\ref{rigid}, $f_d$ is rigid over $X^n$.
It follows that for some open neighbourhood $U$ of $d$ in $D$, $f_U:\widetilde Z_U\to U\times X^n$ is biholomorphic over $U\times X^n$ with $\id_U\times f_d:U\times\widetilde Z_d\to U\times X^n$.
In particular, for all $s\in U$, $Z_s=f_s(\widetilde Z_s)=f_d(\widetilde Z_d)=Z_d$.
The universality of the Douady space now implies that $U=\{d\}$, so that in fact $D=\{d\}$, as desired.
\end{proof}

%\bibliographystyle{plain}
%\bibliography{../ccs}

\end{document}